\documentclass[12pt]{article}
\usepackage[utf8]{inputenc}
\usepackage[OT2, T1]{fontenc}
\usepackage{amssymb,amsmath}
\usepackage{epsfig}
\usepackage{enumerate}
\usepackage{titlesec}
\usepackage{titletoc}
\usepackage{amsthm}
\usepackage{stmaryrd}
\usepackage{enumitem}
\usepackage{graphicx}
\usepackage{semantic}
\usepackage{amscd}
\usepackage{amsmath}
\usepackage{amsfonts}
\usepackage{mathrsfs}
\usepackage{amssymb}
\usepackage{verbatim}
\usepackage{color}
\usepackage[linktoc=page, colorlinks=true, allcolors=blue]{hyperref}
\usepackage{mathtools}
\usepackage{braket}
\usepackage{multicol}
\usepackage{url}
\usepackage{etoolbox}
\usepackage{geometry}
\usepackage{verbatim}
\usepackage{anysize}

\usepackage{tikz}\usepackage{pgfplots}
\pgfplotsset{compat=1.15}
\usepackage{mathrsfs}
\usetikzlibrary{arrows}
\usepackage{graphicx}
\marginsize{1.7cm}{1.7cm}{0.5cm}{1.0cm}
\usepackage[linewidth=1pt]{mdframed}
\theoremstyle{plain}

\theoremstyle{definition}
\newtheorem{lis}{List}[section]

\newtheorem{remark}{Remark}[section]
\theoremstyle{definition}
\newtheorem{al}{Algorithm}[section]
\newtheorem{definition}{Definition}[section]
\newtheorem{example}{Example}[section]

\newtheorem{proposition}{Proposition}[section]

\usepackage{array}
\usepackage[a]{esvect}
\def\t{\hbox{\scalebox{0.75}{$\triangle$}}}

\def\S{S_{\triangle}}

\def\a{a}

\def\C{C_{\triangle}}

\def\S{S_{\triangle}}

\usepackage{hyperref,xcolor}
\hypersetup{
    colorlinks=true,
    linkcolor=blue,
    urlcolor=blue,
    }

\begin{document}

\title{Automated Generation of Triangle Geometry Theorems}

\author{Alexander Skutin\thanks{This work was supported by the Ministry of Education and Science of the Russian Federation as part of the
program of the Moscow Center for Fundamental and Applied Mathematics under the agreement no. 075-15-2022-284, by the scholarship of Theoretical Physics and Mathematics Advancement Foundation “BASIS” (grant No 21-8-3-2-1) and by the Russian Science
Foundation, project no. 22-11-00075.}}

\date{}
\maketitle
\begin{abstract}
     In this article, we introduce an algorithm for automatic generation and categorization of triangle geometry theorems.
\end{abstract}

\section{Introduction}\label{sc3}

Plane geometry is a vast field of research where many theorems had been obtained and new results are still being discovered. Over the past few decades, a lot of effort has been spent on creating algorithms designed to automatically generate theorems in plane geometry, some of which can be found in \cite{1, 2, 8, 3, 4}.

In this paper, we concretize the problem of automatic generation of plane geometry theorems for the case of triangle geometry theorems, that is, triangle $ABC$ theorems that are invariant with respect to permutations of $ABC$ vertices. We provide a new algorithm that generates and categorize triangle geometry theorems. It is expected that this algorithm is able to generate almost all of the theorems from the articles \cite{cos, cos1}. The main idea of our algorithm can be described as follows:\\The algorithm has inductive form and at each new step $t$
\begin{enumerate}

\item it considers a set of theorems obtained on the previous step and constructs a new set of theorems by adding at most one new object to each already existed theorem and formulating new theorems about the resulting configurations,
\item it replaces the set of obtained theorems with some of its ``maximal generalizations''.

\end{enumerate}

The definition of ``maximally general'' (complete) sets of theorems will be presented in this article.

\subsection{Notation}\label{circ}

The arity $\text{ar}(f)$ of a function $f$ is the number of variables acting in $f$. Further, by $\wedge, \Rightarrow, \Leftrightarrow$ we will denote the logical operators `and', `implies' and `equivalent'. We will use the standard set-theory notation $\{x \:\vert\: \text{statement about x}\}$ which is read as, ``the set of all x such that the statement about x is true.''

\subsection{Structure of the paper} The paper is organized as follows. In Sections 2, 3 we introduce $\t$-objects and define the set $S_{\triangle}^{7}$. In Section 4, we develop an algorithm for automatic generation of triangle geometry theorems based on $S_{\triangle}^{7}$. Section 5 contains some propositions that simplify the computation of $S_{\triangle}^{7}$. Appendix A contains lists of objects which are used in the article.

\subsection{Triangle centers and lines}
\begin{definition}[C. Kimberling, \cite{ki}]\label{d5}
By a {\em triangle center} $ X$ denote a point $ X(A, B, C)$, which is defined for each tuple of points $A$, $B$, $C$ on the plane $\mathbb{R}^2$.
\end{definition}

\begin{lis}\label{l1}
The complete list of triangle centers $X_i$, $1\leq i\leq 13$ that are used in this article can be found in the Appendix (see List \ref{ltc} in the Appendix). Some of the centers in use with corresponding numbers:
\begin{enumerate}

\item \text{In(ex)center} $I$, $I(A, B, C)$ -- the incenter of $ABC$ if $A$, $B$, $C$ are placed clockwise on the plane $\mathbb{R}^2$ (or the $A$-excenter of $ABC$ if $A$, $B$, $C$ are placed anti-clockwise on $\mathbb{R}^2$).

\item \text{Centroid} $G$, $G(A, B, C)$ -- the centroid of $ABC$.

\item \text{Circumcenter} $O$, $O(A, B, C)$ -- the circumcenter of $ABC$.

\item \text{Orthocenter} $H$, $H(A, B, C)$ -- the orthocenter of $ABC$.

\item \text{Nine-point center} $N$, $N(A, B, C)$ -- the nine-point center of $ABC$.

\item[7.] \text{First(second) Fermat point} $F$, $F(A, B, C)$ -- the first Fermat point of $ABC$ if $A$, $B$, $C$ are placed clockwise on the plane $\mathbb{R}^2$ (or the second Fermat point of $ABC$ if $A$, $B$, $C$ are placed anti-clockwise on $\mathbb{R}^2$).

\item[9.] \text{Inner(outer) Feuerbach point} $F_e$, $F_e(A, B, C)$ -- the inner Feuerbach point of $ABC$ if $A$, $B$, $C$ are placed clockwise on the plane $\mathbb{R}^2$ (or the $A$-external Feuerbach point of $ABC$ if $A$, $B$, $C$ are placed anti-clockwise on $\mathbb{R}^2$).

\item[12.] \text{Inner(outer) Morley point} $M$, $M(A, B, C)$ -- the $A$-vertex of the inner Morley triangle of $ABC$ if $A$, $B$, $C$ are placed clockwise on the plane $\mathbb{R}^2$ (or the $A$-vertex of the outer Morley triangle of $ABC$ if $A$, $B$, $C$ are placed anti-clockwise on $\mathbb{R}^2$).

\end{enumerate}
\end{lis}

\newpage \section{Definitions of \texorpdfstring{$\t$}{t}-objects}

\begin{definition}\label{d10}
Denote by a {\em $\t$-point} any 6-tuple of points lying on the plane $\mathbb{R}^2$.
\end{definition}

\begin{remark}
    Similarly, one can introduce $\t$-lines, $\t$-circles and other $\t$-curves, but we omit them in this article for simplicity.
\end{remark}

\begin{definition}
For each $\t$-point $x = (x_1, x_2, x_3, x_4, x_5, x_6)$ define $$x_{bc} = x_1,\:\: x_{cb} = x_2,\:\: x_{ca} = x_3,\:\: x_{ac} = x_4,\:\: x_{ab} = x_5,\:\: x_{ba} = x_6.$$
\end{definition}

\begin{example}\label{lam}
    Consider the Van Lamoen configuration (see \cite{van}) -- a triangle $ABC$ with the centroid $G$, the cevian triangle $A'B'C'$ of $G$ wrt $ABC$ and the circumcenters $O_{bc} = O(GBC'),\ldots, $\\$O_{ba} = O(GBA')$ of $GBC',\ldots, GBA'$. In this configuration it is possible to define the following $\t$-points $x = (A, A, B, B, C, C)$, $y = (G, \ldots , G)$, $z = (A', A', B', B', C', C')$, $t = (O_{bc}, \ldots , O_{ba})$. Thus, $x_{bc} = A,\ldots, x_{ba} = C$, $y_{bc} = G,\ldots, y_{ba} = G$, $z_{bc} =A',\ldots, z_{ba} = C'$, $t_{bc} = O_{bc},\ldots, t_{ba} = O_{ba}$.
\end{example}

\begin{definition}\label{d15}

Denote by a {\em $\t$-function} any function $f$ which corresponds a non-empty set of $\t$-points to each $\text{ar}(f)$-tuple of $\t$-points, and is one of the functions $f_{\triangle, i}$ which are listed in the Appendix of this article (see List \ref{ltf} in the Appendix).\\Some of the $\t$-functions in use with corresponding numbers:\\ (these are $\t$-functions which will be used in the further definitions and examples)\begin{enumerate}
\item[1.] $f_{\triangle, 1} = $ the set of all $\t$-points ($f_{\triangle, 1}$ has arity $0$ and, thus, is a set of $\t$-points. Same can be said about $f_{\triangle, i}$, $1\leq i\leq 8$).
\item[2.] $f_{\triangle, 2} = \{x\text{ is a }\t\text{-point} \:\vert\: x_{bc} = x_{cb}, x_{ca} = x_{ac}, x_{ab} = x_{ba}\}$.

    \item[8.] $f_{\triangle, 8} = \{x\text{ is a }\t\text{-point}\:\vert\: x_{bc},\ldots, x_{ba}\text{ lie on a circle}\}$.

\item[11.] $f_{\triangle, 11}( x) = \{y\text{ is a }\t\text{-point}\:\vert\: y_{bc} = x_{bc},\ldots, y_{ba} = x_{ba}\text{ i.e. }y = x\}$.

\item[17.] $f_{\triangle, 17}( x,  y, z) = \left\{t\text{ is a }\t\text{-point}\left\vert\begin{array}{cl} t_{bc},\ldots, t_{ba}\text{ are the projections of}\\x_{bc}, \ldots, x_{ba}\text{ on }y_{bc}z_{bc},\ldots, y_{ba}z_{ba}\end{array}\right.\right\}$.

\item[19.] $f_{\triangle, {19, i}}( x,  y, z) = \left\{t\text{ is a }\t\text{-point}\left\vert\begin{array}{cl} t_{bc} =  X_i(x_{bc}, y_{bc}, z_{bc}), t_{cb} =  X_i(x_{cb}, z_{cb}, y_{cb}),\hfill\hfill\\ t_{ca} =  X_i(z_{ca}, x_{ca}, y_{ca}), t_{ac} =  X_i(y_{ac}, x_{ac}, z_{ac}),\hfill\hfill\\ t_{ab} =  X_i(y_{ab}, z_{ab}, x_{ab}), t_{ba} =  X_i(z_{ba}, y_{ba}, x_{ba})\end{array}\right.\right\}$,\\where $1\leq i\leq 13$, $ X_i$ denotes the $i$-th center from the list \ref{l1}.

\item[20.] $f_{\triangle, {20}}( x,  y, z, t) = \left\{v\text{ is a }\t\text{-point}\left\vert\begin{array}{cl} v_{bc} =  x_{bc}y_{bc}\cap z_{bc}t_{bc},\ldots,\\ v_{ba} =  x_{ba}y_{ba}\cap z_{ba}t_{ba}\hfill\hfill\end{array}\right.\right\}$.

\item[25.] Functions of the form $f_{\triangle, n, \alpha, \beta, \gamma}(x, y, z) := f_{\triangle, n}(x^{\alpha}, y^{\beta}, z^{\gamma})$, $1\leq n\leq 24$, where $\alpha, \beta, \gamma$ are any symbols from the set $\{bc, cb, ca, ac, ab, ba\}$ and for each $\t$-point $x$,\\$x^{bc} := (x_{bc}, x_{cb}, x_{ca}, x_{ac}, x_{ab}, x_{ba}),\quad x^{ac} := (x_{ac}, x_{ca}, x_{cb}, x_{bc}, x_{ba}, x_{ab}),$\\$x^{cb} := (x_{cb}, x_{bc}, x_{ba}, x_{ab}, x_{ac}, x_{ca}),\quad x^{ba} := (x_{ba}, x_{ab}, x_{ac}, x_{ca}, x_{cb}, x_{bc}),$\\$x^{ab} := (x_{ab}, x_{ba}, x_{bc}, x_{cb}, x_{ca}, x_{ac}),\quad x^{ca} := (x_{ca}, x_{ac}, x_{ab}, x_{ba}, x_{bc}, x_{cb})$,\\ denotes the orbit of $x$.

\end{enumerate}
\end{definition}

\begin{definition}\label{N}
    Consider the sequence $x_1, x_2, x_3, \ldots$ of free variables which can be any $\t$-points. Denote by a {\em $\t$-configuration} any logical statement about the sequence $x_1, x_2, x_3,\ldots$, which has the form $$\bigwedge_{i = 1}^N[x_{a_i}\in f_i(x_{b_{i,1}}, x_{b_{i, 2}}, \ldots, x_{b_{i,\text{ar} (f_i)}})], $$where\begin{enumerate}
        \item $N$ is a natural number

        \item $a_1 < a_2 < \ldots < a_N$ is a strongly increasing sequence of natural numbers
        
        \item for each $1\leq i\leq N$, $f_i$ is a $\t$-function
        
        \item for each $1\leq i\leq N$, $b_{i, 1}, b_{i, 2},\ldots, b_{i,\text{ar} (f_i)} < a_i$ is an $\text{ar} (f_i)$-tuple of natural numbers $<a_i$.

        \end{enumerate}
    
\end{definition}
Since statements of the form $[x_i\in f_{\triangle, 1}]$ don't carry any additional information, we will omit such terms within $\t$-configurations (i.e. we may not consider the $\t$-function $f_{\triangle, 1}$).

\begin{definition}

For any $\t$-configuration $c$ denote by $\text{deg}(c)$, $\text{height}(c)$ the values of $N$ and $a_N$ from the definition \ref{N} which are related to $c$, respectively.
\end{definition}

\begin{example}\label{s}
    Consider the following $\t$-configuration $c$ which is related to the example \ref{lam}$$c = [x_1\in f_{\triangle, 2}]\wedge [x_2\in f_{\triangle, 19, 2}(x_1, x_1^{ab}, x_1^{ac})]\wedge [x_3\in f_{\triangle, 20}(x_1, x_2, x_1^{ab}, x_1^{ac})]\wedge [x_5\in f_{\triangle, 19, 3}(x_1, x_2, x_3^{ab})].$$So $\text{deg}(c) = 4$, $\text{height}(c) = 5$.
\end{example}

\begin{definition}
    Let $C_{\triangle}$ denote the set of all $\t$-configurations. Also for each natural $n$ let $\C^n$ denote the set of all $\t$-configurations $c$ with $\text{height}(c)\leq n$. We will say that $c, d\in\C$ are {\em equivalent} if there exists a permutation of variables $\sigma : x_1, x_2, x_3,\ldots\to x_1, x_2, x_3,\ldots$ which sends $c$ to $d$, i.e. $\sigma(c) = d$. We will label $c\simeq d$ for each equivalent $c, d\in\C$.
\end{definition}

\begin{definition}

For each $\t$-configuration $$c = \bigwedge_{i = 1}^N[x_{a_i}\in f_i(x_{b_{i,1}}, x_{b_{i, 2}}, \ldots, x_{b_{i,\text{ar} (f_i)}})], $$let $\text{terms}(c)$ denote the set of $\text{deg} = 1$ $\t$-configurations $$\text{terms}(c) := \{[x_{a_i}\in f_i(x_{b_{i,1}}, x_{b_{i, 2}}, \ldots, x_{b_{i,\text{ar} (f_i)}})]\:\vert\: 1\leq i\leq N\}.$$Also we will say that $d\in\C$ is a {\em predecessor} of $c$ if $$d = \bigwedge_{i = 1}^M[x_{a_i}\in f_i(x_{b_{i,1}}, x_{b_{i, 2}}, \ldots, x_{b_{i,\text{ar} (f_i)}})]$$for some $1\leq M\leq N$.
\end{definition}

\begin{definition}
    For each $\t$-configurations $c, c_1, c_2,\ldots, c_l$ we say that $c = \cup_{i = 1}^lc_i = c_1\cup\ldots\cup c_l$ if $\text{terms}(c) = \cup_{i = 1}^l\text{terms}(c_i) = \text{terms}(c_1)\cup\ldots\cup\text{terms}(c_l)$. Also for each $\t$-configurations $c, d$ we say that $c \subseteq d$ if $\text{terms}(c) \subseteq \text{terms}(d)$, and $c \subsetneq d$ if $\text{terms}(c) \subsetneq \text{terms}(d)$.
\end{definition}

\begin{definition}
    For each $\t$-configurations $c, d$ we say that $c \leq d$ if there exist $\t$-configurations $c'\simeq c, d'\simeq d$, which are equivalent to $c, d$ respectively and are such that $c'$ is a predecessor of $d'$.

\end{definition}

\begin{definition}\label{conf}
    Denote by a {\em $\t$-theorem} any valid implication of the form $c \Rightarrow r$, $c, r\in\C$, where $\text{deg}(r) = 1$.
\end{definition}

\begin{example}

Consider the $\t$-configuration $c$ as in the example \ref{s}, and let $r = [x_5\in f_{\triangle, 8}]$. Then from the Van Lamoen theorem (see \cite{van}) we have that $c \Rightarrow r$ is a $\t$-theorem.
\end{example}

\begin{definition}\label{abc}
	Consider a triangle $ABC$ lying on the plane $\mathbb{R}^2$ in general position. For each $\t$-configuration $$c = \bigwedge_{i = 1}^N[x_{a_i}\in f_i(x_{b_{i,1}}, x_{b_{i, 2}}, \ldots, x_{b_{i,\text{ar} (f_i)}})]\in\C, $$ let $c(ABC)$ denote\begin{enumerate}
	    
     \item the set of $\t$-points $\{x_{a_1}, x_{a_2},\ldots, x_{a_N}\}$ satisfying the system of equations $$\left\{\begin{array}{cl}x_{a_1} = (A, A, B, B, C, C)\in f_1(x_{b_{1,1}}, x_{b_{1, 2}}, \ldots, x_{b_{1,\text{ar} (f_1)}})\hfill\hfill\\ x_{a_2}\in f_2(x_{b_{2,1}}, x_{b_{2, 2}}, \ldots, x_{b_{2,\text{ar} (f_2)}})\hfill\hfill\\\ldots\hfill\hfill\\ x_{a_N}\in f_N(x_{b_{N,1}}, x_{b_{N, 2}}, \ldots, x_{b_{N,\text{ar} (f_N)}})\hfill\hfill\end{array}\right.$$if $f_1 = f_{\triangle, 2}$, $a_i = i$ ($1\leq i\leq N$), and this system of equations has the unique solution
     \item $c(ABC) = \varnothing$, otherwise.
	\end{enumerate}

\end{definition}

\begin{example}
    Consider the following $\t$-configuration $c'$$$c' = [x_1\in f_{\triangle, 2}]\wedge [x_2\in f_{\triangle, 19, 2}(x_1, x_1^{ab}, x_1^{ac})]\wedge [x_3\in f_{\triangle, 20}(x_1, x_2, x_1^{ab}, x_1^{ac})]\wedge [x_4\in f_{\triangle, 19, 3}(x_1, x_2, x_3^{ab})].$$So $c'$ is equivalent to $c$ from the example \ref{s}, and $c'(ABC) \not= \varnothing$.
\end{example}

\begin{definition}
For each triangle $ABC$ in general position and each set $S\subseteq \C$, denote $$S(ABC) := \bigcup_{c\in S}c(ABC).$$
\end{definition}

\section{Construction of \texorpdfstring{$S_{\triangle}^7$}{T7S7}}

\begin{definition}
    For a $\t$-configuration $c$ denote by $\normalfont{\text{Gen}(c)}$ the set of {\em generalizations of $c$}, where $$\normalfont{\text{Gen}(c)} \!:=\! \left\{d\in\C\left\vert\begin{array}{cl} \text{there exist }\t\text{-configurations }c'\simeq c, d'\simeq d,\text{ which are equivalent to }c, d\\\text{respectively and are such that: }\hfill\hfill\\\:1.\:\: c'\Rightarrow d'\text{ is a valid implication, and}\hfill\hfill\\\:2.\:\: c'\not\Leftrightarrow d' \hfill\hfill\end{array}\right.\right\}$$
\end{definition}

\begin{remark}
    It is also possible to implement a larger set of generalizations of $c$ by further considering cases when $d' = \sigma(d)$ for some surjective (and not necessarily bijective) mapping of variables $\sigma : x_1, x_2, x_3,\ldots\to x_1, x_2, x_3,\ldots$, and adding some additional condition 3 (see for example the generalization of Gergonne theorem in \cite[Theorem 9.1(1), p.13]{SNT}). However, we will omit such generalizations for simplicity.
\end{remark}

\begin{definition}
    A set $S\subseteq\C$ is called {\em complete} (we will also call such a set as ``maximally general'') if for each $d\in S$ the set of all $\t$-theorems of the form $c\Rightarrow r$, $c\in S$, $r\in\C$ can't be deductively derived\footnote{Here by ``can be deductively derived'' we mean ``can be derived with using the set of inference rules:

\begin{equation}
\inference{(a, b, c\text{ are any logical statements}) & a  \Rightarrow  b & b  \Rightarrow  c}
{a  \Rightarrow  c}
\end{equation}

\begin{equation}
\inference{(a, b, c\text{ are any logical statements})}
{a\Leftrightarrow a\wedge a\qquad a\wedge b  \Rightarrow  a\qquad a\wedge b  \Rightarrow  b\qquad a\wedge b \Leftrightarrow b\wedge a\qquad (a\wedge b)\wedge c \Leftrightarrow a\wedge (b\wedge c)}
\end{equation}

\begin{equation}
\inference{(a_1, a_2, b_1, b_2\text{ are any logical statements}) & a_1  \Rightarrow  b_1 & a_2  \Rightarrow  b_2}
{a_1\wedge a_2  \Rightarrow  b_1\wedge b_2}
\end{equation}

\begin{equation}
\inference{(c, d\in\C, \sigma : x_1, x_2, x_3,\ldots\to x_1, x_2, x_3,\ldots\text{ is a surjective mapping of variables}) & c \Rightarrow d }
{\sigma (c) \Rightarrow \sigma (d)}
\end{equation}

\begin{equation}
\inference{(c_1, c_2, d\in\C) & c_1\Rightarrow d & c_2\text{ is a predecessor of } c_1 & \text{height}(d)\leq\text{height}(c_2)}
{c_2 \Rightarrow d}.
\end{equation}

} from the set of $\t$-theorems of the form $c\Rightarrow r$, $c\in (S\setminus\{d\})\cup\text{Gen}(d)$, $r\in\C$.
\end{definition}

\begin{definition}
    Consider a natural number $n$. A set $S\subseteq\C$ is called {\em $n$-complete} if it is complete and each $\t$-theorem of the form $c\Rightarrow r$, $c\in\C^n$, $r\in\C$ can be deductively derived$^1$ from the set of $\t$-theorems of the form $c\Rightarrow r$, $c\in S$, $r\in\C$.
\end{definition}

Next, we will be interested in computing $7$-complete sets, however none of these sets can be computed in practice, and we finish this section by constructing its computable analogue $\S^7$.

\begin{definition}\label{cgen}
	A $\t$-theorem $c \Rightarrow r$ is called {\em computably generalizable} if there exists a set of $\t$-theorems of the form $\{c_i \Rightarrow r_i, d \Rightarrow r\:\vert\: 1\leq i\leq l\}$, such that $l\geq 1$, $c_i\subsetneq c$, $d = \cup_{i = 1}^l r_i$, $d\not\Leftrightarrow c$ ($1\leq i\leq l$). Obviously each computably generalizable $\t$-theorem $c\Rightarrow r$ can be deductively derived$^1$ from the set of $\t$-theorems $\{c_i \Rightarrow r_i, d \Rightarrow r\:\vert\: 1\leq i\leq l\}$ and, thus, from the set of $\t$-theorems of the form $c'\Rightarrow r'$, $c'\in\text{Gen}(c)$, $r'\in\C$.
\end{definition}

\begin{definition}\label{cccgen}
	Consider a triangle $ABC$ in general position. For a $\t$-configuration $c$ denote by $\normalfont{\text{Gen}_{\triangle}(c)}$ the set of {\em triangular computable generalizations of $c$}, where $$\normalfont{\text{Gen}_{\triangle}(c)} := \{d\in\C\:\vert\: d(ABC)\not=\varnothing, d\leq c,\text{ and } d\not\simeq c\}.$$Also for a $\t$-configuration $c$, define \begin{equation*}
	\text{CGen}_{\triangle}(c) := \begin{cases}\text{Gen}_{\triangle}(c), &\text{if each }\t\text{-theorem of the form }c \Rightarrow r, r\in\C \text{ with}\\&\text{height}(r)\leq \text{height}(c)\text{ is computably generalizable},\\c, &\text{otherwise}.\end{cases}\end{equation*}Additionally, for a set $S\subseteq \C$, define the sets\begin{itemize}
		\item  $\text{CGen}_{\triangle}(S) := \cup_{c\in S}\text{CGen}_{\triangle}(c)$, 
		\item $\text{CGen}_{\triangle}^1(S) := \text{CGen}_{\triangle}(S)$,
		
		\item $\text{CGen}_{\triangle}^{i + 1}(S) := \text{CGen}_{\triangle}(\text{CGen}_{\triangle}^i(S))$, $i = 1, 2, 3,\ldots$,
		
		\item $\text{MaxCGen}_{\triangle}(S) := \text{CGen}_{\triangle}^d(S)$, where $d$ is the minimal natural number such that $\text{CGen}_{\triangle}^d(S) = \text{CGen}_{\triangle}^{d + 1}(S)$.
	\end{itemize}
\end{definition}

\begin{remark}
    The sets $\text{CGen}_{\triangle}(S)$, $\text{CGen}_{\triangle}^i(S)$, $\text{MaxCGen}_{\triangle}(S)$ can be calculated in practice with the help of Propositions \ref{jjjj}, \ref{jj} from Section 5.
\end{remark}

\begin{definition}\label{d}
Define the sets $\S^n\subseteq \C^n$, $n\geq 1$ inductively. Let $\S^1 := \{[x_1\in f_{\triangle, 2}]\}$. Assume that for a natural $t\geq 1$ the set $\S^t\subseteq\C^t$ is already constructed. Consider the sets \begin{enumerate}
    \item[] $J := \left\{c\in\C^{t + 1}\left\vert\begin{array}{cl} c = c_1\cup c_2\in\C^{t + 1}\text{ for some}\hfill\hfill\\c_1\in \S^t, c_2\in\C^{t + 1}\text{ with }\text{deg}(c_2) = 1, \text{height}(c_2) = t + 1\end{array}\right.\right\},$
    
    \item[] $\S^{t + 1} := \S^t\cup\text{MaxCGen}_{\triangle}(J).$
    
\end{enumerate}From this inductive process construct the sets $\S^1\subseteq \S^2\subseteq \ldots\subseteq \S^n\subseteq\ldots$. It is easy to see that $\S^n = \text{MaxCGen}_{\triangle}(\S^n)$ for each $n\geq 1$.
\end{definition}

In what follows, we will be interested in computing the set $\S^7$. The set $\S^7$ can be seen as a computable analogue and an approximation of $7$-complete sets.

\begin{remark}\label{rlc}
    Note that when calculating $\S^7$, on each new step, we don't need to list those $\t$-configurations that have already been listed.
\end{remark}
The set $S_{\triangle}^{7}$ can be computed in practice from its definition with the help of Remark \ref{rlc} and the Propositions \ref{jjjj}, \ref{jj} from Section 5.

\section{Automated generation of theorems based on \texorpdfstring{$S_{\triangle}^{7}$}{S7}}
In this section, we introduce an algorithm for a computer that generates and categorizes triangle geometry theorems based on the set $S_{\triangle}^{7}$.

\begin{definition}
    For each $\t$-point $a\in S_{\triangle}^{7}(ABC)$ let $f_1(a), f_2(a), \ldots, f_{\gamma(a)}(a)$ denote the sequence of $\t$-functions which are used for the definition of $a$ and are ordered according their appearance. Also denote by $\Gamma(a)$ the sequence $(f_1(a), f_2(a),\ldots, f_{\gamma(a)}(a))$ after excluding those $f_k(a), 1\leq k\leq \gamma(a)$ which do not have the form $f_{\triangle, 19, i}$ for some $1\leq i\leq 13$, and then replacing each uniform segment of the remaining sequence of $\t$-functions with the single $\t$-function of the same type (for example, if $a\in S_{\triangle}^{7}(ABC)$ is such that $(f_1(a), f_2(a), \ldots, f_{\gamma(a)}(a)) = (f_{\triangle, 2}, f_{\triangle, 19, 1}, f_{\triangle, 3}, f_{\triangle, 19,1}, f_{\triangle, 19, 9})$, then $\Gamma(a) = (f_{\triangle, 19, 1}, f_{\triangle, 19, 9})$).
\end{definition}

The next algorithm generates triangle theorems based on the computation of $S_{\triangle}^{7}(ABC)$. Also it produces an intuitive categorization of theorems, the same as in the articles \cite{cos, cos1}.

\begin{al}\label{d425}
	The computer program inputs a sequence $X_{i_1}, X_{i_2}, \ldots, X_{i_d}$, $1\leq i_1, i_2,\ldots, i_d\leq 13$, $d\geq 1$ of triangle centers from the list \ref{l1} which has no uniform segments (i.e. $i_k\not=i_{k + 1}$, $1\leq k < d$), and then produces the output after following the steps below.
	\begin{enumerate}

 \item For a triangle $ABC$ in general position, compute the set $\S^7(ABC)$.
 
\item In ``Objects $X_{i_1}- X_{i_2}- \ldots- X_{i_d}$'' section print the definitions and notations of all $\t$-points $a\in S_{\triangle}^{7}(ABC)$ with $\Gamma(a) = (f_{\triangle, 19, i_1}, f_{\triangle, 19, i_2},\ldots, f_{\triangle, 19, i_d})$.
As in the ETC \cite{etc}, we can label objects from the section ``Objects $X_{i_1}- X_{i_2}- \ldots- X_{i_d}$'' as $(X_{i_1}- X_{i_2}- \ldots- X_{i_d})_1$, $(X_{i_1}- X_{i_2}- \ldots- X_{i_d})_2$, $(X_{i_1}- X_{i_2}- \ldots- X_{i_d})_3$, $\ldots$.

\item Compute and print in ``Properties $X_{i_1}- X_{i_2}- \ldots- X_{i_d}$'' section all correct statements of the form $[a\in f(a_1, a_{2}, \ldots, a_{\text{ar} (f)})]$, where $f$ is any $\t$-function and $a$, $a_{1}$, $a_{2}$, $\ldots$, $a_{\text{ar} (f)}$ are any $\t$-points from the section ``Objects $X_{i_1}- X_{i_2}- \ldots- X_{i_d}$''.
	\end{enumerate}
\end{al}
\begin{remark}
Note that in the section ``Objects $X_{i_1}- X_{i_2}- \ldots- X_{i_d}$'' from the algorithm \ref{d425} for each object $x$ it is possible to leave only the representative $(x_{bc}, \ldots, x_{ba})$ of the orbit of elements$$x = x^{bc} := (x_{bc}, x_{cb}, x_{ca}, x_{ac}, x_{ab}, x_{ba}),\quad x^{ac} := (x_{ac}, x_{ca}, x_{cb}, x_{bc}, x_{ba}, x_{ab}),$$$$x^{cb} := (x_{cb}, x_{bc}, x_{ba}, x_{ab}, x_{ac}, x_{ca}),\quad x^{ba} := (x_{ba}, x_{ab}, x_{ac}, x_{ca}, x_{cb}, x_{bc}),$$$$x^{ab} := (x_{ab}, x_{ba}, x_{bc}, x_{cb}, x_{ca}, x_{ac}),\quad x^{ca} := (x_{ca}, x_{ac}, x_{ab}, x_{ba}, x_{bc}, x_{cb}).$$We also need to replace all sequences $x, y,\ldots, z$ of objects from ``Objects $X_{i_1}- X_{i_2}- \ldots- X_{i_d}$'' that have the same coordinates as $\t$-points (i.e. are such that $x = y = \ldots = z$) on the single object $x$, and list all descriptions of $x$, that are coming from $x, y, \ldots, z$, in the definition of $x$.
\end{remark}

\subsection{Relation to the articles [8, 9]}We expect that a computer program based on the algorithm \ref{d425} will be able to generate almost all of the theorems from the articles \cite{cos, cos1}.

\section{Practical implementation}\label{sss}

\begin{definition}
Consider any $\t$-configuration $$c = \bigwedge_{i = 1}^N[x_{a_i}\in f_i(x_{b_{i,1}}, x_{b_{i, 2}}, \ldots, x_{b_{i,\text{ar} (f_i)}})]\in\C. $$ For each $l\geq 1$, denote by $O_l(c)$ the set of all $\t$-configurations $c'$ which have the following form $$c' = \left(\bigwedge_{\substack{1\leq i\leq N \: :\: a_i< l}}[x_{a_i}\in f_i(x_{b_{i,1}}, x_{b_{i, 2}}, \ldots, x_{b_{i,\text{ar} (f_i)}})]\right)\wedge$$$$\wedge [x_{l}\in f(x_{b_{1}}, x_{b_{ 2}}, \ldots, x_{b_{\text{ar} (f)}})]\wedge$$$$\wedge\left(\bigwedge_{\substack{1\leq i\leq N \: :\: a_i > l}}[x_{a_i}\in f_i(x_{b_{i,1}}, x_{b_{i, 2}}, \ldots, x_{b_{i,\text{ar} (f_i)}})]\right)$$for some $\text{deg} = 1$, $\text{height} = l$ $\t$-configuration $[x_{l}\in f(x_{b_{1}}, x_{b_{ 2}}, \ldots, x_{b_{\text{ar} (f)}})]$.
\end{definition}

The following propositions \ref{jjjj}, \ref{jj} can be used for computing $\text{Gen}_{\triangle}(\cdot)$, $\text{CGen}_{\triangle}(\cdot)$, \\$\text{MaxCGen}_{\triangle}(\cdot)$, and $\S^7$ from the definitions \ref{cccgen}, \ref{d}.

For a $\t$-configuration $c\in\C^7$, the set $\{e\in\C^7\:\vert\: e\leq c, e\not\simeq c,\text{ and } e(ABC)\not=\varnothing\}$ can be easily computed in practice by brutal force method, thus to compute $\text{Gen}_{\triangle}(c)$ it is enough to develop a method for calculating the set $\{d\in\C^7\:\vert\: c\Leftrightarrow d\}$. The next proposition \ref{jjjj} describes such a method.

\begin{proposition}\label{jjjj}
Consider any $\t$-configuration $c\in\C^7$. The set $\{d\in\C^7\:\vert\: c\Leftrightarrow d\}$ can be computed after providing the following steps:\begin{enumerate}

    \item consider the set $D_1$ of all $d_1\in O_{7}(c)$ with $d_1 \Leftrightarrow c$
    \item consider the set $D_2$ of all $d_2\in \cup_{d_1\in D_1}O_{6}(d_1)$ with $d_2 \Leftrightarrow c$
    \item repeat step 3 for $D_2$ instead of $D_1$ and finish with the set $D_3 = \{d_3\in \cup_{d_2\in D_2}O_{5}(d_2)\:\vert\: d_3 \Leftrightarrow c\}$
    \item repeat step 3 for $D_3, D_4,\ldots$ until we finish with the set $D_7$ which satisfies $D_7 = \{d\in\C^7\:\vert\: c\Leftrightarrow d\}$.
\end{enumerate}

\end{proposition}

To compute the sets $\text{CGen}_{\triangle}(\cdot)$, $\text{MaxCGen}_{\triangle}(\cdot)$, we need to use the method of computation of $\text{Gen}_{\triangle}(c)$, which was described previously, and also to develop a method for checking whether a given $\t$-theorem $c \Rightarrow r$ with $c, r\in\C^7$ is computationally generalizable. The next proposition \ref{jj} describes such a method.

\begin{proposition}\label{jj}
Consider any $\t$-theorem $c\Rightarrow r$ such that $c, r\in\C^7$. Then to understand whether $c \Rightarrow r$ is computationally generalizable we need to provide the following steps:\begin{enumerate}
\item Compute the set $$U_c := \{d\in\C^7\:\vert\: \text{deg}(d) = 1\text{ and }c' \Rightarrow d,\text{ for some }c'\subsetneq c\}$$
    \item  consider the set $D_1$ of all $d_1\in O_{7}(c)$ with $d_1 \Rightarrow r$ and $\text{terms}(d_1)\subseteq U_c$
    \item if there exists $d_1\in D_1$ with $d_1\not\Leftrightarrow c$, then finish with the string ``$c \Rightarrow r$ is computationally generalizable''. Otherwise consider the set $D_2$ of all $d_2\in \cup_{d_1\in D_1}O_{6}(d_1)$ with $d_2 \Rightarrow r$ and $\text{terms}(d_2)\subseteq U_c$
    \item repeat step 3 for $d_2\in D_2$ instead of $d_1\in D_1$ and finish either with the string ``$c \Rightarrow r$ is computationally generalizable'', or with the set $$D_3 = \{d_3\in \cup_{d_2\in D_2}O_{5}(d_2)\:\vert\: d_3 \Rightarrow r, \text{terms}(d_3)\subseteq U_c\}$$
    \item repeat step 3 for $D_3, D_4,\ldots$ until we finish either with the string ``$c \Rightarrow r$ is computationally generalizable'', or with the set $D_7$, and in the latter case return the string ``$c \Rightarrow r$ is not computationally generalizable''.
\end{enumerate}

\end{proposition}

Propositions \ref{jjjj}, \ref{jj} are trivial consequences of the following Proposition \ref{j}.

\begin{proposition}\label{j}
Consider any $\t$-configurations $$c = \bigwedge_{i = 1}^N[x_{a_i}\in f_i(x_{b_{i,1}}, x_{b_{i, 2}}, \ldots, x_{b_{i,\text{ar} (f_i)}})]\in\C, $$ $$c' = \bigwedge_{i = 1}^{N'}[x_{a_i'}\in f_i'(x_{b_{i,1}'}, x_{b_{i, 2}'}, \ldots, x_{b_{i,\text{ar} (f_i')}'})]\in\C, $$ with $c\Rightarrow c'$. Then we have that for each natural $l$, $\t$-configurations\begin{enumerate}
    \item $c(l) = \displaystyle{\bigwedge_{\substack{i = 1\\a_i\leq l}}^N[x_{a_i}\in f_i(x_{b_{i,1}}, x_{b_{i, 2}}, \ldots, x_{b_{i,\text{ar} (f_i)}})]}$
    \item $c'(l) = \displaystyle{\bigwedge_{\substack{i = 1\\a_i'\leq l}}^{N'}[x_{a_i'}\in f_i'(x_{b_{i,1}'}, x_{b_{i, 2}'}, \ldots, x_{b_{i,\text{ar} (f_i')}'})]}$
    \item $\displaystyle{c''(l) = \bigwedge_{\substack{i = 1\\a_i'> l}}^{N'}[x_{a_i'}\in f_i'(x_{b_{i,1}'}, x_{b_{i, 2}'}, \ldots, x_{b_{i,\text{ar} (f_i)}'})]}$
    \item $d(l) = c(l)\wedge c''(l)$
    \end{enumerate}
are such that $c(l)\Rightarrow c'(l)$, $c\Rightarrow d(l)\Rightarrow c'$.

\end{proposition}\begin{proof}Proposition \ref{j} follows from the fact that for each $l\geq 1$ and each $\t$-configuration $$c = \bigwedge_{i = 1}^N[x_{a_i}\in f_i(x_{b_{i,1}}, x_{b_{i, 2}}, \ldots, x_{b_{i,\text{ar} (f_i)}})]\in\C, $$variables $x_{i}$ inside $c$ with $1\leq i\leq l$ are independent of variables $x_j$ inside $c$ with $j > l$.\end{proof}

\section{Appendix}
This appendix contains the complete lists of triangle centers and $\t$-functions that we use in this article.

\begin{lis}\label{ltc}

The list of triangle centers

\begin{enumerate}

\item \text{In(ex)center} $I$, $I(A, B, C)$ -- the in center of $ABC$ if $A$, $B$, $C$ are placed clockwise on the plane $\mathbb{R}^2$ (or the $A$-excenter of $ABC$ if $A$, $B$, $C$ are placed anti-clockwise on $\mathbb{R}^2$).

\item \text{Centroid} $G$, $G(A, B, C)$ -- the centroid of $ABC$.

\item \text{Circumcenter} $O$, $O(A, B, C)$ -- the circumcenter of $ABC$.

\item \text{Orthocenter} $H$, $H(A, B, C)$ -- the orthocenter of $ABC$.

\item \text{Nine-point center} $N$, $N(A, B, C)$ -- the nine-point center of $ABC$.

\item \text{Symmedian point} $S$, $S(A, B, C)$ -- the Symmedian point of $ABC$.

\item \text{First(second) Fermat point} $F$, $F(A, B, C)$ -- the first Fermat point of $ABC$ if $A$, $B$, $C$ are placed clockwise on the plane $\mathbb{R}^2$ (or the second Fermat point of $ABC$ if $A$, $B$, $C$ are placed anti-clockwise on $\mathbb{R}^2$).

\item \text{First(second) Isodynamic point} $I_s$, $I_s(A, B, C)$ -- the first isodynamic point of $ABC$ if $A$, $B$, $C$ are placed clockwise on the plane $\mathbb{R}^2$ (or the second Isodynamic point of $ABC$ if $A$, $B$, $C$ are placed anti-clockwise on $\mathbb{R}^2$).

\item \text{Inner(outer) Feuerbach point} $F_e$, $F_e(A, B, C)$ -- the inner Feuerbach point of $ABC$ if $A$, $B$, $C$ are placed clockwise on the plane $\mathbb{R}^2$ (or the $A$-external Feuerbach point of $ABC$ if $A$, $B$, $C$ are placed anti-clockwise on $\mathbb{R}^2$).

\item \text{Euler reflection point} $E$, $E(A, B, C)$ -- the Euler Reflection point of $ABC$.

\item \text{Inner(outer) Apollonian point} $A_p$, $A_p(A, B, C)$ -- the $A$-vertex of inner Apollonian triangle of $ABC$ if $A$, $B$, $C$ are placed clockwise on the plane $\mathbb{R}^2$ (or the $A$ -- vertex of outer Apollonian triangle of $ABC$ if $A$, $B$, $C$ are placed anti-clockwise on $\mathbb{R}^2$).

\item \text{Inner(outer) Morley point} $M$, $M(A, B, C)$ -- the $A$-vertex of inner Morley triangle of $ABC$ if $A$, $B$, $C$ are placed clockwise on the plane $\mathbb{R}^2$ (or the $A$-vertex of outer Morley triangle of $ABC$ if $A$, $B$, $C$ are placed anti-clockwise on $\mathbb{R}^2$).

\item \text{Isogonal point} $Iso$, $Iso(A, B, C, D)$ -- the Isogonal conjugation of $D$ wrt $ABC$.

\item Other similar triangular centers and lines.

\end{enumerate}

\end{lis}

\begin{lis}\label{ltf}

The list of $\t$-functions

\begin{enumerate}
\item $f_{\triangle, 1} = $ the set of all $\t$-points ($f_{\triangle, 1}$ has arity $0$ and, thus, is a set of $\t$-points. Same can be said about $f_{\triangle, i}$, $1\leq i\leq 8$).
\item $f_{\triangle, 2} = \{x\text{ is a }\t\text{-point} \:\vert\: x_{bc} = x_{cb}, x_{ca} = x_{ac}, x_{ab} = x_{ba}\}$.
\item $f_{\triangle, 3} = \left\{x\text{ is a }\t\text{-point} \left\vert\begin{array}{cl} x_{bc} = x_{cb}, x_{ca} = x_{ac}, x_{ab} = x_{ba}\hfill\hfill\\\text{and the triangle }x_{bc}x_{ca}x_{ab}\text{ is equilateral}\end{array}\right.\right\}$.
\item $f_{\triangle, 4} = \{x\text{ is a }\t\text{-point}\:\vert\: x_{bc} = x_{ca} = x_{ab}, x_{cb} = x_{ac} = x_{ba}\}$.
\item $f_{\triangle, 5} = \{x\text{ is a }\t\text{-point}\:\vert\: x_{bc} = \ldots = x_{ba}\}$.
    
\item $f_{\triangle, 6} = \{x\text{ is a }\t\text{-point}\:\vert\:\text{points }x_{bc},\ldots, x_{ba}\text{ are collinear}\}$.
    \item $f_{\triangle, 7} = \{x\text{ is a }\t\text{-point}\:\vert\: x_{bc},\ldots, x_{ba}\text{ lie on a conic}\}$.
    \item $f_{\triangle, 8} = \{x\text{ is a }\t\text{-point}\:\vert\: x_{bc},\ldots, x_{ba}\text{ lie on a circle}\}$.
    \item $f_{\triangle, 9}( x) = \{y\text{ is a }\t\text{-point}\:\vert\: \text{lines }x_{bc}y_{bc}, \ldots, x_{ba}y_{ba}\text{ are concurrent}\}$.
\item $f_{\triangle, 10}( x) = \left\{y\text{ is a }\t\text{-point}\left\vert\begin{array}{cl} \text{the midpoints of segments}\hfill\hfill\\x_{bc}y_{bc}, \ldots, x_{ba}y_{ba}\text{ are collinear}\hfill\hfill\end{array}\right.\right\}$.

\item $f_{\triangle, 11}( x) = \{y\text{ is a }\t\text{-point}\:\vert\: y_{bc} = x_{bc},\ldots, y_{ba} = x_{ba}\text{ i.e. }y = x\}$.

\item $f_{\triangle, 12}( x) = \left\{y\text{ is a }\t\text{-point}\left\vert\begin{array}{cl} y_{bc} = y_{cb}, y_{ca} = y_{ac}, y_{ab} = y_{ba}\text{ and}\hfill\hfill\\\text{the triangle }x_{bc}x_{cb}\cap x_{ca}x_{ac}\cap x_{ab}x_{ba}\\\text{is perspective to }y_{bc}y_{ca}y_{ab}\hfill\hfill\end{array}\right.\right\}$.
    \item $f_{\triangle, 13}( x) = \left\{y\text{ is a }\t\text{-point}\left\vert\begin{array}{cl} \text{---/--- }x_{bc}x_{cb}\cap x_{ca}x_{ac}\cap x_{ab}x_{ba}\\\text{is orthologic to }y_{bc}y_{ca}y_{ab}\hfill\hfill\end{array}\right.\right\}$.

\item $f_{\triangle, 14}( x) = \left\{y\text{ is a }\t\text{-point}\left\vert\begin{array}{cl} y_{bc} = \ldots = y_{ba}\text{ and }y_{bc}\text{ lies on the}\\\text{circumcircle of the triangle}\hfill\hfill\\x_{bc}x_{cb}\cap x_{ca}x_{ac}\cap x_{ab}x_{ba}\hfill\hfill\end{array}\right.\right\}$.

\item $f_{\triangle, 15}( x,  y) = \left\{z\text{ is a }\t\text{-point}\left\vert\begin{array}{cl} z_{bc},\ldots, z_{ba}\text{ coincides with the}\hfill\hfill\\\text{midpoints of }x_{bc}y_{bc},\ldots, x_{ba}y_{ba}\text{ resp.}\hfill\hfill\end{array}\right.\right\}$.
\item $f_{\triangle, 16}( x,  y) = \left\{z\text{ is a }\t\text{-point}\left\vert\begin{array}{cl} z_{bc},\ldots, z_{ba}\text{ lie on the lines}\\x_{bc}y_{bc},\ldots, x_{ba}y_{ba},\text{ resp.}\hfill\hfill\end{array}\right.\right\}$.

\item $f_{\triangle, 17}( x,  y, z) = \left\{t\text{ is a }\t\text{-point}\left\vert\begin{array}{cl} t_{bc},\ldots, t_{ba}\text{ are the projections}\hfill\hfill\\\text{of }x_{bc}, \ldots, x_{ba}\text{ on }y_{bc}z_{bc},\ldots, y_{ba}z_{ba}\end{array}\right.\right\}$.

\item $f_{\triangle, 18}( x,  y, z) = \left\{t\text{ is a }\t\text{-point}\left\vert\begin{array}{cl} t_{bc},\ldots, t_{ba}\text{ are the reflections}\hfill\hfill\\\text{of }x_{bc},\ldots, x_{ba}\text{ wrt }y_{bc}z_{bc},\ldots, y_{ba}z_{ba}\end{array}\right.\right\}$.

\item $f_{\triangle, {19, i}}( x,  y, z) =\left\{t\text{ is a }\t\text{-point}\left\vert\begin{array}{cl} t_{bc} =  X_i(x_{bc}, y_{bc}, z_{bc}), t_{cb} =  X_i(x_{cb}, z_{cb}, y_{cb}),\hfill\hfill\\ t_{ca} =  X_i(z_{ca}, x_{ca}, y_{ca}), t_{ac} =  X_i(y_{ac}, x_{ac}, z_{ac}),\hfill\hfill\\ t_{ab} =  X_i(y_{ab}, z_{ab}, x_{ab}), t_{ba} =  X_i(z_{ba}, y_{ba}, x_{ba})\end{array}\right.\right\}$,\\where $1\leq i\leq 12$, $ X_i$ denotes the $i$-th center from the list \ref{l1}.

\item $f_{\triangle, {20}}( x,  y, z, t) = \left\{v\text{ is a }\t\text{-point}\left\vert\begin{array}{cl} v_{bc} =  x_{bc}y_{bc}\cap z_{bc}t_{bc},\ldots,\\ v_{ba} =  x_{ba}y_{ba}\cap z_{ba}t_{ba}\hfill\hfill\end{array}\right.\right\}$.

\item $f_{\triangle, 21}( x,  y, z, t) = \left\{\begin{array}{cl}v\text{ is a}\\\t\text{-point}\end{array}\left\vert\begin{array}{cl} v_{bc} =  X_{13}(x_{bc}, y_{bc}, z_{bc}, t_{bc}), v_{cb} =  X_{13}(x_{cb}, z_{cb}, y_{cb}, t_{cb}), \hfill\hfill\\v_{ca} =  X_{13}(z_{ca}, x_{ca}, y_{ca}, t_{ca}), v_{ac} =  X_{13}(y_{ac}, x_{ac}, z_{ac}, t_{ac}), \\v_{ab} =  X_{13}(y_{ab}, z_{ab}, x_{ab}, t_{ab}), v_{ba} =  X_{13}(z_{ba}, y_{ba}, x_{ba}, t_{ba})\hfill\hfill\end{array}\right.\right\}$,\\where $ X_{13}$ denote the $13$-th center from the list \ref{l1}.

\item $f_{\triangle, 22}(x, y, z, t) = \left\{v\text{ is a }\t\text{-point}\left\vert\begin{array}{cl}\text{points }v_{bc},\ldots, v_{ba}\text{ lie on the pivotal}\hfill\hfill\\\text{isocubics of triangles}\hfill\hfill\\x_{bc}y_{bc}z_{bc},\ldots, x_{ba}y_{ba}z_{ba}\text{ with pivots}\\t_{bc},\ldots, t_{ba},\text{ respectively}\hfill\hfill\end{array}\right.\right\}$.

\item $f_{\triangle, 23}(x, y, z) = \left\{t\text{ is a }\t\text{-point}\:\vert\: x_{bc}y_{bc}z_{bc}t_{bc},\ldots, x_{ba}y_{ba}z_{ba}t_{ba}\text{ are cyclic}\right\}$.

\item $f_{\triangle, 24}( x, y) = \left\{z\text{ is a }\t\text{-point}\left\vert\begin{array}{cl} z_{bc}\text{ lies on the rectangular hyperbola}\\\text{passing through the vertices of the}\hfill\hfill\\\text{triangle }x_{bc}x_{cb}\cap x_{ca}x_{ac}\cap x_{ab}x_{ba}\hfill\hfill\\\text{and the point }y_{bc},\hfill\hfill\\\text{and similarly for }z_{cb},\ldots, z_{ba}\hfill\hfill\end{array}\right.\right\}$.

\item Functions of the form $f_{\triangle, n, \alpha, \beta, \gamma}(x, y, z) := f_{\triangle, n}(x^{\alpha}, y^{\beta}, z^{\gamma})$, $1\leq n\leq 24$, where $\alpha, \beta, \gamma$ are any symbols from the set $\{bc, cb, ca, ac, ab, ba\}$ and for each $\t$-point $x$,\\$x^{bc} := (x_{bc}, x_{cb}, x_{ca}, x_{ac}, x_{ab}, x_{ba}),\quad x^{ac} := (x_{ac}, x_{ca}, x_{cb}, x_{bc}, x_{ba}, x_{ab}),$\\$x^{cb} := (x_{cb}, x_{bc}, x_{ba}, x_{ab}, x_{ac}, x_{ca}),\quad x^{ba} := (x_{ba}, x_{ab}, x_{ac}, x_{ca}, x_{cb}, x_{bc}),$\\$x^{ab} := (x_{ab}, x_{ba}, x_{bc}, x_{cb}, x_{ca}, x_{ac}),\quad x^{ca} := (x_{ca}, x_{ac}, x_{ab}, x_{ba}, x_{bc}, x_{cb})$,\\ denotes the orbit of $x$.

\item Other similar functions $f_{\triangle, i}$ and, for example, we can consider

$f_{\triangle, i} = \left\{x\text{ is a }\t\text{-point}\left\vert\begin{array}{cl} x_{bc}x_{ca}x_{ab}\text{ is similar}\\\text{(perspective, orthologic) to }x_{cb}x_{ac}x_{ba}\end{array}\right.\right\}$.

\end{enumerate}

\end{lis}

\addcontentsline{toc}{section}{Bibliography}

\bibliographystyle{unsrt}

\end{document}